\PassOptionsToPackage{unicode}{hyperref}
\PassOptionsToPackage{hyphens}{url}
\documentclass[
]{article}
\usepackage{lmodern}
\usepackage{amssymb,amsmath}
\usepackage{ifxetex,ifluatex}
\ifnum 0\ifxetex 1\fi\ifluatex 1\fi=0 
  \usepackage[T1]{fontenc}
  \usepackage[utf8]{inputenc}
  \usepackage{textcomp} 
\else 
  \usepackage{unicode-math}
  \defaultfontfeatures{Scale=MatchLowercase}
  \defaultfontfeatures[\rmfamily]{Ligatures=TeX,Scale=1}
\fi
\IfFileExists{upquote.sty}{\usepackage{upquote}}{}
\IfFileExists{microtype.sty}{
  \usepackage[]{microtype}
  \UseMicrotypeSet[protrusion]{basicmath} 
}{}
\makeatletter
\@ifundefined{KOMAClassName}{
  \IfFileExists{parskip.sty}{%
    \usepackage{parskip}
  }{
    \setlength{\parindent}{0pt}
    \setlength{\parskip}{6pt plus 2pt minus 1pt}}
}{
  \KOMAoptions{parskip=half}}
\makeatother
\usepackage{xcolor}
\IfFileExists{xurl.sty}{\usepackage{xurl}}{} 
\IfFileExists{bookmark.sty}{\usepackage{bookmark}}{\usepackage{hyperref}}
\hypersetup{
  pdftitle={An analogue of the relationship between SVD and pseudoinverse over double-complex matrices},
  pdfauthor={Ran Gutin (Department of Computer Science, Imperial College London)},
  hidelinks,
  pdfcreator={LaTeX via pandoc}}
\usepackage{amsthm} \newtheorem{definition}{Definition}[section] \newtheorem{lemma}{Lemma}[section] \newtheorem{theorem}{Theorem}[section] \newtheorem{corollary}{Corollary}[theorem]  
\ifluatex
  \usepackage{selnolig}  
\fi

\title{An analogue of the relationship between SVD and pseudoinverse
over double-complex matrices}
\author{Ran Gutin \thanks{Department of Computer Science, Imperial College London}}
\date{\today}

\begin{document}

\maketitle
\begin{abstract}
We present a generalisation of the pseudoinverse operation to pairs of
matrices, as opposed to single matrices alone. We note the fact that the
Singular Value Decomposition can be used to compute the ordinary
Moore-Penrose pseudoinverse. We present an analogue of the Singular
Value Decomposition for pairs of matrices, which we show is inadequate
for our purposes. We then present a more sophisticated analogue of the
SVD which includes features of the Jordan Normal Form, which we show is
adequate for our purposes. This analogue of the SVD, which we call the
Jordan SVD, was already presented in a previous paper by us called
``Matrix decompositions over the double numbers''. We adopt the
idea presented in that same paper that a pair of matrices is actually a
single matrix over the double-complex number system. 
\textbf{AMS}: 1502, 13P25
\end{abstract}

\hypertarget{introduction}{%
\section{Introduction}\label{introduction}}

It is not uncommon in linear algebra to study generalisations of matrix
decompositions to pairs of matrices as opposed to single matrices alone
(\cite{producteigenvalue}, see section 7.7.2 of \cite{computations}). In a previous
paper called \emph{Matrix decompositions over the double numbers}
(\cite{gutin}), Gutin suggested a formalism for studying such decompositions.
The suggestion is that a pair of matrices, denoted \([A,B]\), is
actually a single matrix over the double or double-complex number
systems. If \(A\) and \(B\) are real matrices, then we should use the
double numbers. If \(A\) and \(B\) are instead complex matrices,
then we should use the double-complex number system. The double-complex numbers are a direct generalisation of the double numbers.

For the purposes of studying generalisations of Singular Value
Decomposition to this setting, it is most convenient to use the
double-complex number system. Using the double numbers is also
possible, if less convenient, as we indicate later on. From now on, we
refer to a pair of complex matrices of the form \([A,B]\) as a matrix
over the double-complex number system.

Gutin's paper references this one. In his paper, he introduces a
decomposition of double-complex matrices called the \emph{Jordan SVD}. Since
the notion is a new one, Gutin suggests some applications of the
idea. One application presented in that paper is for challenging (and
very likely refuting) a claim found in the English translation of
Yaglom's \emph{Complex numbers in geometry} (\cite{yaglom1968complex})
concerning the group of Linear Fractional Transformations over the
double number system. Gutin also references the main result that
we are going to prove in this paper. This result states roughly that all
the applications of the ordinary SVD to computing the pseudoinverse can
be generalised to the double-complex number system if we replace the ordinary
SVD with the Jordan SVD.

The generalisation of the pseudoinverse operation which we are going to
use is going to be defined by the usual four axioms, but interpreted
over double-complex matrices. Since those axioms are arguably unintuitive and
quite unwieldly, we give an alternative ``diagrammatic'' definition
which we prove is equivalent to the axiomatic one. This definition has
the advantage that it reduces the pseudoinverse operation over double-complex
matrices to fairly natural operations on complex matrices.

A theme in Gutin's paper is that a verbatim generalisation of a matrix
operation to double-complex matrices usually reduces to familiar matrix
operations over complex matrices. The pseudoinverse operation appears to
be no exception.

We show that whenever the pseudoinverse exists, then so does the Jordan
SVD. Note that both operations are sometimes invalid over some double-complex
matrices. This has two possible applications:

\begin{itemize}
\item
  The Jordan SVD can be used to compute the pseudoinverse of a double-complex
  matrix, in a similar way to how the SVD can be wielded to compute the
  pseudoinverse of a complex matrix.
\item
  We get a sufficient condition for the Jordan SVD to exist, which is
  strong, but not quite a necessary condition.
\end{itemize}

The computational advantages of the Jordan SVD can be debated. The
Jordan decomposition is rarely used in numerical applications,
apparently for reasons of numerical instability. The connection is still
obviously an interesting one.

We motivate further study of the Jordan SVD by the fact that it generalises many matrix decompositions, sometimes directly as in the case of the SVD and Jordan decomposition, and sometimes indirectly as in the case of the polar decomposition for complex matrices and Algebraic Polar Decomposition. The Algebraic Polar Decomposition is a variant of the polar decomposition for complex matrices which eliminates all references to complex conjugation, and doesn't exist for some matrices (\cite{algebraicpolar,kaplansky}). The verbatim generalisation of the polar decomposition to double-complex matrices is in fact equivalent to the Jordan SVD, in the sense that it exists for exactly the same matrices, and one can be derived from the other.

Finally, we remark on a possible misunderstanding surrounding the ring
\(\mathbb C \oplus \mathbb C\). All four arithmetic operations over this
ring (addition, subtraction, multiplication, division) happen
componentwise. It might then be mistakenly concluded that there is
nothing here to study. However, linear
algebra over the complex numbers also includes the \emph{complex
conjugation} operation. A case could be made that a good
generalisation of complex conjugation over arbitrary rings is an
arbitrary ring involution. The complex numbers themselves admit
\emph{two} involutions in fact, one of which is complex conjugation, and
the other one is the trivial involution \(z \mapsto z\). Each involution
over a ring gives rise to a different kind of linear algebra. Notions
like unitary matrices and Hermitian matrices change meaning if the
involution changes. Over the double-complex numbers, any involution which is
componentwise results in a trivial theory. Out of 6 possible involutions,
four of them are componentwise operations:
\begin{itemize}
  \item $(w,z) \mapsto (w,z)$
  \item $(w,z) \mapsto (\overline w,z)$
  \item $(w,z) \mapsto (w,\overline z)$
  \item $(w,z) \mapsto (\overline w,\overline z)$
\end{itemize}
These all result in a trivial theory. This then leaves two that are not
componentwise:
\begin{itemize}
  \item $(w,z) \mapsto (z,w)$
  \item $(w,z) \mapsto (\overline z, \overline w)$
\end{itemize}
These two involutions result in the same theory. We pick the first, but
we could've just as easily picked the second, and it wouldn't make a
difference.

\hypertarget{definitions}{%
\section{Definitions}\label{definitions}}

\hypertarget{sec:definitions}{%
\subsection{Double numbers and double
matrices}\label{sec:definitions}}

The \emph{double numbers}, called split-complex numbers in Wikipedia, and
\emph{perplex numbers} in some other sources) are the hypercomplex
number system in which every number is of the form

\[a + bj, a \in \mathbb R, b \in \mathbb R\]

where \(j\) is a number that satisfies \(j^2 = 1\) without \(j\) being
either \(1\) or \(-1\).

The definition of the arithmetic operations should be clear from the
above description. Define \((a + bj)^* = a - bj\) (similar to complex
conjugation).

Consider the basis \(e = \frac{1+j}2\) and \(e^* = \frac{1-j}2\).
Observe that \(e^2 = e\), \((e^*)^2 = e^*\) and \(e e^* = e^* e = 0\).
Thus \((a e + b e^*) (c e + d e^*) = (ac) e + (bd) e^*\). In other
words, multiplication of double numbers is componentwise in this
basis. Thanks to this, we see that the double numbers can be
defined as the algebra \(\mathbb R \oplus \mathbb R\), where \(\oplus\)
denotes the direct sum of algebras. Observe that
\((a e + b e^*)^* = b e + a e^*\).

We shall consider matrices of double numbers. A matrix \(M\) of
double numbers can be written in the form
\(A \frac{1 + j}2 + B^T \frac{1 - j}2\) where \(A\) and \(B\) are real
matrices. Write \([A,B]\) for \(A \frac{1 + j}2 + B^T \frac{1 - j}2\).
Observe the following identities: \[\begin{aligned}
\\
[A,B] + [C,D] &= [A+C,B+D] \\
[A,B] \times [C,D] &= [AC, DB] \\
[A,B]^* &= [B,A]
\end{aligned}\]

In the above \([A,B]^*\) refers to the conjugate-transpose operation.
Notice that the second component of \([A,B] \times [C,D]\) is \(DB\),
not \(BD\).

\hypertarget{double-complex-numbers-and-matrices}{%
\subsection{Double-complex numbers and
matrices}\label{double-complex-numbers-and-matrices}}

A \emph{double-complex number} (\cite{doublecomplex1}) is a number of the form
\(w + zj\) where \(w\) and \(z\) are complex numbers, and \(j^2 = +1\).
Multiplication is the same as over the double numbers, namely:
\((w + zj)(w' + z'j) = ww' + zz' + j(wz' + zw')\). The product is a
commutative one. The definition of conjugation is
\((w + zj)^* = w - zj\). Sometimes these numbers are called
\emph{tessarines}.

By a \emph{double-complex matrix}, we mean a matrix whose entries are
double-complex numbers. A double-complex matrix can always be expressed in the
form \([A,B]\), which we define to mean
\(A \frac{1+j}{2} + B^T \frac{1-j}{2}\) where \(A\) and \(B\) are
complex matrices. Observe that we use the tranpose of \(B\) rather than
the conjugate-transpose of \(B\).

Observe then the following identities\label{basic-operations}: \[\begin{aligned}
\\
[A,B] + [C,D] &= [A+C,B+D] \\
[A,B] \times [C,D] &= [AC, DB] \\
[A,B]^* &= [B,A]
\end{aligned}\]

\hypertarget{sec:families}{%
\subsection{Families of double-complex matrices}\label{sec:families}}

We say that a double-complex matrix \(M\) is \emph{Hermitian} if it
satisfies \(M^* = M\). Observe then that a Hermitian matrix is precisely
of the form \([A,A]\).

We say that a double-complex matrix \(M\) is \emph{unitary} if it
satisfies \(M^* M = I\). Observe then that a unitary matrix is precisely
of the form \([A,A^{-1}]\).

We say that a double-complex matrix \(M\) is \emph{lower triangular} if
it is of the form \([L,U]\) where \(L\) is a lower triangular real
matrix and \(U\) is an upper triangular real matrix. Similarly, a matrix
of the form \([U,L]\) is called \emph{upper triangular}.

A double-complex matrix is called \emph{diagonal} if it is of the form
\([D,E]\) where \(D\) and \(E\) are diagonal real matrices.

A double-complex matrix is complex precisely when it is of the form
\([A,A^T]\).

\hypertarget{linear-map-notation}{%
\subsection{Linear map notation}\label{linear-map-notation}}

When giving the diagrammatic characterisation of the pseudoinverse, we
find it helpful to treat complex matrices as linear maps. When
introducing a linear map, we may do so using the notation
\(f: U \to V\). We call \(U\) the \emph{domain} of \(f\) and \(V\) the
\emph{codomain} of \(f\). We will denote the \emph{composition} of two
linear maps \(f\) and \(g\) as \(f \circ g\), where
\((f \circ g)(x) = f(g(x))\).

Sometimes, we write \(f(W)\) where \(f\) is some linear map or matrix
and \(W\) is either a vector space or a set. When we write this, we are
referring to the set \(\{f(w) : w \in W\}\). Sometimes, this set will
itself be a vector space.

\section{The role of involutions}

The automorphism group of the double-complex algebra (understood as an algebra over the real numbers) has eight elements. Six of those have order $2$, and can therefore be called \emph{involutions}. Involutions are central to the study of double-complex matrices because they can be used as a substitute or even generalisation of complex conjugation. Linear algebra essentially reduces to just five operations: Addition, subtraction, multiplication, division, and conjugation.

As we stated in the introduction, four of those involutions result in a trivial theory. They are:
\begin{itemize}
  \item $(w,z) \mapsto (w,z)$
  \item $(w,z) \mapsto (\overline w,z)$
  \item $(w,z) \mapsto (w,\overline z)$
  \item $(w,z) \mapsto (\overline w,\overline z)$
\end{itemize}
To illustrate why they result in a trivial theory, we pick one of them, and hope that the reader understands that a similar issue is present in the other three. By $(w,z)$, we mean $w \frac{1+j}2 + z\frac{1-j}2$ where $w$ and $z$ are complex numbers.

For the sake of argument, we will define $(w,z)^*$ to mean $(\overline w,\overline z)$. We can define $(A,B)$ to mean $(1,0) A + (0,1) B$ where $A$ and $B$ are arbitrary complex matrices. What we observe is that if $f(M)$ is some matrix operation defined on complex matrices $M$, we can generalise its definition to double-complex matrices by interpreting any complex conjugations in its definition to mean the operation $(w,z)^* = (\overline w, \overline z)$. We will then immediately have that $f((A,B)) = (f(A),f(B))$. For instance, if $f$ is the pseudoinverse of a matrix (its verbatim generalisation that uses this particular involution) then we have that this operation always distributes over the components $A$ and $B$. The same thing is true if $f$ is the SVD or the LU decomposition, or we would argue \emph{anything}. There is therefore nothing to study.

We also mentioned two interesting involutions in the introduction. They are:
\begin{itemize}
  \item $(w,z) \mapsto (z,w)$
  \item $(w,z) \mapsto (\overline z, \overline w)$
\end{itemize}
We claimed that these somehow result in an identical theory to eachother. To understand why, we endeavour to define an operation $(A,B) \mapsto [A,B]$ on pairs of complex matrices, mapping to the double-complex matrices bijectively, that satisfies:
$$\begin{aligned}
[A,B] + [C,D] &= [A+C,B+D] \\
[A,B] \times [C,D] &= [AC, DB] \\
[A,B]^* &= [B,A]
\end{aligned}$$
The third identity involves an involution, so our definition of $[A,B]$ will change along with the involution. If the involution is defined as $(w,z) \mapsto (z,w)$, then we can define $[A,B]$ to mean $A \frac{1+j}2 + B^T\frac{1-j}2$ as we have chosen to do in this paper. This then causes the above three identities to all be satisfied. Likewise, if we define our involution to be $(w,z) \mapsto (\overline z, \overline w)$, then we can define $[A,B]$ to mean $A \frac{1+j}2 + \overline B^T \frac{1-j}2$. Once again, the above three identities get satisfied.

So what's the difference between these two involutions? It is about which double-complex matrices $[A,B]$ are understood to be the complex ones. One approach makes it so that the double-complex matrices of the form $[A,A^T]$ are taken to be the complex ones. The other approach makes it so that double-complex matrices of the form $[A,A^*]$ are taken to be the complex ones. Since in this paper, we treat double-complex matrices as pairs of complex matrices, we don't ultimately care which ones are taken to be complex.\footnote{With the exception of lemma \ref{orthonormal}, where we bypassed the $[A,B]$ notation. This can still be shown to be valid regardless of how the involution is defined.} By accident, we made it so that the complex matrices are those of the form $[A,A^T]$.

Observe that given either definition of $[A,B]$, the claim that a double-complex matrix is Hermitian precisely when it satisfies $A = B$ is still true. All the other characterisations of the different families of matrices given in section \ref{sec:families} stay the same.

The convention that matrices of the form $[A,\overline A^T]$ are the complex one (which is the one we didn't pick) is arguably slightly better. It results in the notions of unitary or Hermitian double-complex matrices being direct generalisations of their complex counterparts. For that reason, in future papers, this convention might be a slightly better one to adopt.

Given that it turns out that our notion of Hermitian matrix is not a direct generalisation of the one over complex matrices, what does it generalise? It turns out that it includes all the \emph{complex symmetric} matrices, which are those complex matrices which satisfy $A = A^T$. Our notion of unitary matrix generalises the set of \emph{complex orthogonal} matrices, which are those that satisfy $A^T A = AA^T = I$. This all follows from section \ref{sec:families}. Perhaps this suggests that complex orthogonality and unitarity are on equal footings with eachother.

\hypertarget{singular-value-decomposition-over-double-complex-matrices}{%
\section{Singular Value Decomposition over double-complex
matrices}\label{singular-value-decomposition-over-double-complex-matrices}}

In the following, we work with square matrices. It is possible to
generalise to non-square matrices.

Consider the singular value decomposition (\cite{strang09}): \[M = U S V^*\]
where \(U\) and \(V\) are unitary, and \(S\) is diagonal and Hermitian.
If we interpret this over the double-complex numbers (using section
\ref{sec:families}), we get \[[A,B] = [P,P^{-1}] [D,D] [Q^{-1},Q]\] We break
into components and get \[A = PDQ^{-1}, B = QDP^{-1}.\] Finally, observe
that \(AB = P D^2 P^{-1}\) and \(BA = Q D^2 Q^{-1}\). In other words, we
get the familiar eigendecompositions of \(AB\) and \(BA\).

Finally, notice that while it's obvious that SVD is reducible to
eigendecomposition (namely, the eigendecompositions of \(A^* A\) and
\(A A^*\)) it is less obvious that a reduction in the opposite direction
is possible. We have provided such a reduction here.

\hypertarget{introduction-to-the-jordan-svd}{%
\section{Introduction to the Jordan
SVD}\label{introduction-to-the-jordan-svd}}

Given the above derivation of double-complex SVD, it's clear that not all
double-complex matrices have an SVD. This follows from the fact that we need
\(AB\) and \(BA\) to be diagonalisable. We propose a partial remedy to
this problem that uses the Jordan Normal Form. For the sake of
simplicity, we will assume that all our matrices are square.

\begin{definition}[Jordan SVD]
The Jordan SVD of a double-complex matrix $M$ is a factorisation of the form $M = U [J,J] V^*$ where $J$ is a complex Jordan matrix, and $U$ and $V$ are unitary double-complex matrices.
\end{definition}

All non-singular double-complex matrices have a Jordan SVD. This is an
improvement over the naive double-complex SVD, which doesn't always exist
even for non-singular matrices. We will prove a stronger result, which
states that any double-complex matrix with a pseudoinverse has a Jordan SVD. Take
note though that not all double-complex matrices have a Jordan SVD. Some
necessary conditions for a double-complex matrix \([A,B]\) to have a Jordan
SVD are that:

\begin{enumerate}
\def\labelenumi{\arabic{enumi}.}

\item
  \(\operatorname{rank}(A) = \operatorname{rank}(B)\)
\item
  \(\sqrt{AB}\) exists
\item
  \(\sqrt{BA}\) exists
\end{enumerate}

All three conditions are necessary, and not one of them can be dropped
because the other two are insufficient. For instance, the matrix
\([A,B]\) where
\(A = \left(\begin{matrix}0 & 0\\0 & 1\end{matrix}\right)\) and
\(B = \left(\begin{matrix}0 & 1\\0 & 0\end{matrix}\right)\) does not
have a Jordan SVD. This is because conditions 1 and 2 are satisfied, but
condition 3 is not.

Finding a necessary and sufficient condition for a Jordan SVD to exist
remains an open problem.

\subsection{As a generalisation of other matrix decompositions}

In this section, we attempt to show how the Jordan SVD relates to polar decomposition and SVD. We also connect it to attempts to define analogues of SVD and polar decomposition of complex matrices where the role of complex conjugation is replaced with the trivial involution (\cite{algebraicpolar,kaplansky}). The results here are mainly useful for motivation, and are not needed in later section.

The Jordan SVD generalises the Jordan decomposition. Let $PJP^{-1}$ be the Jordan decomposition of a complex matrix $A$. We then have that $[A,A] = U [J,J] V^*$ where $U = V = [P,P^{-1}]$. Note that $[A,A]$ is the general form of a double-complex Hermitian matrix.

Observe that the Jordan SVD generalises the ordinary SVD of a complex matrix. A Jordan SVD of $[A,\overline A^T]$ is $[U,\overline U^T][D,D][V,\overline V^T]^*$ where $U$ is a unitary complex matrix, $D$ is a diagonal real matrix, and $V$ is a unitary complex matrix. 

\begin{definition} The \emph{polar decomposition} of a double-complex matrix $M$ is a factorisation of the form $M = UP$ where $U$ is unitary and $P$ is Hermitian. \end{definition}

\begin{theorem} A double-complex matrix $M$ has a Jordan SVD if and only if it has a polar decomposition. \end{theorem}
\begin{proof}
Assume that $M$ has a Jordan SVD. Write it as $W[J,J]V^*$. Let $U = WV^*$ and $P = V[J,J]V^*$. Observe that $UP$ is a polar decomposition of $M$.

We now prove the converse. Assume that $M$ has a polar decomposition $UP$. Write $P$ as $[A,A]$. $A$ has a Jordan decomposition $QJQ^{-1}$. We thus have that $M = U[Q,Q^{-1}][J,J][Q,Q^{-1}]^*$, which is a Jordan SVD where the three factors are $U[Q,Q^{-1}]$, $[J,J]$ and $[Q,Q^{-1}]^*$.
\end{proof}

Observe that the double-complex polar decomposition generalises the so-called \emph{Algebraic Polar Decomposition} (\cite{algebraicpolar,kaplansky}) of complex matrices. The Algebraic Polar Decomposition of a complex matrix $A$ is a factorisation of the form $A = UP$ where $U$ satisfies $UU^T = I$ and $P$ satisfies $P = P^T$. Given $M = [A,A^T]$, we have that a polar decomposition of $M$ is of the form $[U,U^T] [P,P]$.

\hypertarget{double-jordan-svd-as-opposed-to-double-complex-jordan-svd}{%
\subsection{Double-number Jordan SVD as opposed to double-complex Jordan
SVD}\label{double-jordan-svd-as-opposed-to-double-complex-jordan-svd}}

Recall that there is an alternative to the Jordan Normal Form over real
matrices where Jordan blocks of the form
\(\begin{bmatrix}a + bi & 0 \\ 0 & a - bi\end{bmatrix}\) are changed to
blocks of the form \(\begin{bmatrix}a & -b \\ b & a\end{bmatrix}\). This
change results in a variant of the Jordan Normal Form that uses only
real numbers and not complex numbers. Using this variant of the JNF,
it's possible to define a variant of the Jordan SVD for double
matrices without using double-complex numbers. For the sake of completeness,
we define this variant now but don't use it later on.

\begin{definition}[Double-number Jordan SVD]
The double-number Jordan SVD of a double matrix $M$ is a factorisation of the form $M = U [J,J] V^*$ where $J$ is a real Jordan matrix, and $U$ and $V$ are unitary double matrices.
\end{definition}

\hypertarget{introduction-to-the-double-double-complex-pseudoinverse}{%
\section{Foundations of the double / double-complex
pseudoinverse}\label{introduction-to-the-double-double-complex-pseudoinverse}}

Recall the four axioms that define the pseudoinverse \(M^+\) of a matrix
\(M\):

\begin{enumerate}
\def\labelenumi{\arabic{enumi}.}

\item
  \(MM^+M = M\)
\item
  \(M^+MM^+ = M^+\)
\item
  \((M^+M)^*=M^+M\)
\item
  \((MM^+)^*=MM^+\)
\end{enumerate}

The objective of the pseudoinverse is to extend the matrix inverse to
matrices which are singular. We extend these relations to double-complex
matrices to define the pseudoinverse of a double-complex matrix. In this
section, we present another characterisation of the double-complex pseudoinverse which does not use the axioms above. This other characterisation is more intuitive and can be expressed in a clean way using diagrams. This makes proving the last results in this section easier, as well as the results in later sections. We finish this section by proving a necessary and sufficient condition for the pseudoinverse of a double-complex matrix to exist.

First, we demonstrate that not all double-complex matrices have a
pseudoinverse. For instance, consider \(M = \frac{1 + j}2\) (considered
as a \(1 \times 1\) matrix). It's easy to show that this matrix does not have a pseudoinverse by directly appealing to the four axioms above. This fact contrasts with the case of real and complex matrices.
Those matrices do always have pseudoinverses.

The matrix pseudoinverse, when it exists, is always unique. This fact is
true over all rings, including the double-complex numbers. The standard proof
of this fact in the complex case generalises to all rings.

We will begin by giving a necessary condition for the pseudoinverse to exist, which we will later show is a sufficient condition.

\begin{lemma} \label{necessary-condition} If $M = [A,B]$ has a pseudoinverse then $\operatorname{Im}(AB) = \operatorname{Im}(A) \cong \operatorname{Im}(B) = \operatorname{Im}(BA)$. \end{lemma}

\begin{proof}

Let $M^+ = [C,D]$. The four axioms for the pseudoinverse expand out to:

\begin{enumerate}
    \item $ACA = A$ and $BDB = B$
    \item $CAC = C$ and $DBD = D$
    \item $DB = AC$
    \item $BD = CA$
\end{enumerate}

To show that $\operatorname{Im}(A) \cong \operatorname{Im}(B)$, we start by formulating this condition as $\operatorname{rank}(A) = \operatorname{rank}(B)$. This then follows from: $\operatorname{rank}(A) \geq \operatorname{rank}(AC) = \operatorname{rank}(DB) \geq \operatorname{rank}(BDB) = \operatorname{rank}(B) \geq \operatorname{rank}(BD) = \operatorname{rank}(CA) \geq \operatorname{rank}(ACA) = \operatorname{rank}(A)$.

To show that $\operatorname{Im}(AB) = \operatorname{Im}(A)$: $\operatorname{Im}(AB) \subseteq \operatorname{Im}(A) = \operatorname{Im}(ACA) = A(\operatorname{Im}(CA)) = A(\operatorname{Im}(BD)) =  \operatorname{Im}(ABD) \subseteq \operatorname{Im}(AB)$.

To show that $\operatorname{Im}(BA) = \operatorname{Im}(B)$: $\operatorname{Im}(BA) \subseteq \operatorname{Im}(B) = \operatorname{Im}(BDB) = B(\operatorname{Im}(DB)) = B(\operatorname{Im}(AC)) =  \operatorname{Im}(BAC) \subseteq \operatorname{Im}(BA)$.

We are done.
\end{proof}

One theme of this paper is that matrix operations over double
matrices reduce to familiar matrix operations over real matrices. The
axiomatic definition of the double pseudoinverse is quite
complicated and unintuitive, which naturally leads us to look for a
natural meaning to the double pseudoinverse.

Our ``intuitive'' description of the pseudoinverse of \(M = [A,B]\)
begins by assuming that \(A\) and \(B\) can be described using certain
diagrams. This is the diagram for \(A\): \[\begin{aligned}
\operatorname{Im}(B) & \xrightarrow[f]{\sim} & \operatorname{Im}(A) \\
\oplus & & \oplus \\
\ker(A) & \xrightarrow[g]{} & \ker(B) \\
\end{aligned}\] where \(f\) is an isomorphism and \(g(v) = 0\) for every
\(v\).

What do we mean by the above being ``the diagram'' for \(A\)? It means
that given \(v \in \operatorname{Im}(B)\) and \(w \in \ker(A)\) we have
that \(A(v + w) = f(v) + g(w)\). The diagram also indicates, using the
\(\sim\) symbol, that \(f\) is an isomorphism.

This is the diagram for \(B\) (essentially the same as for \(A\)):
\[\begin{aligned}
\operatorname{Im}(A) & \xrightarrow[h]{\sim} & \operatorname{Im}(B) \\
\oplus & & \oplus \\
\ker(B) & \xrightarrow[i]{} & \ker(A) \\
\end{aligned}\] where \(h\) is an isomorphism and \(i(v) = 0\) for every
\(v\). The symbol \(i\) has no special significance; it's simply the
next letter of the alphabet after \(h\).

A question might arise about what happens if \(A\) and \(B\) cannot be
matched to the above two diagrams. In that case, \([A,B]\) does not have
a pseudoinverse.

We now show how to obtain the pseudoinverse from the above two diagrams.
Essentially, you ``reverse'' all the arrows. Let \(M^+ = [C,D]\). This
is the diagram for \(C\): \[\begin{aligned}
\operatorname{Im}(B) & \xleftarrow[f^{-1}]{\sim} & \operatorname{Im}(A) \\
\oplus & & \oplus \\
\ker(A) & \xleftarrow[g^+]{} & \ker(B) \\
\end{aligned}\] where \(g^+(v) = 0\) for every \(v\).

And here is the diagram for \(D\): \[\begin{aligned}
\operatorname{Im}(A) & \xleftarrow[h^{-1}]{\sim} & \operatorname{Im}(B) \\
\oplus & & \oplus \\
\ker(B) & \xleftarrow[i^+]{} & \ker(A) \\
\end{aligned}\] where \(i^+(v) = 0\) for every \(v\).

This completes our description. What follows is a verification that this
description is equivalent to the axiomatic definition of the
pseudoinverse:

\begin{theorem} Let $M = [A,B]$. Assume that $M$ has a pseudoinverse. Let $M^+ = [C,D]$. Then we have that $A$ and $B$ are described by the first two diagrams, and $C$ and $D$ are described by the last two diagrams. \end{theorem}

\begin{proof}

We expand out the four axioms of the pseudoinverse:
\begin{enumerate}
    \item $ACA = A$ and $BDB = B$
    \item $CAC = C$ and $DBD = D$
    \item $DB = AC$
    \item $BD = CA$
\end{enumerate}

We show that $A$ matches the description in the first diagram.

The first step is to show that $V = \operatorname{Im}(B) \oplus \ker(A)$, where $V$ is the vector space we're working over (that is, the domain and codomain of $A$). By the rank-nullity theorem, we have that $\dim(V) = \dim(\operatorname{Im}(A)) + \dim(\ker(A))$. We already know that $\operatorname{Im}(A) \cong \operatorname{Im}(B)$ (by lemma \ref{necessary-condition}), so we have that $\dim(V) = \dim(\operatorname{Im}(B)) + \dim(\ker(A))$. To strengthen that to $V = \operatorname{Im}(B) \oplus \ker(A)$, we need to show that $\operatorname{Im}(B) \cap \ker(A) = \{0\}$. Let $v \in \operatorname{Im}(B) \cap \ker(A)$. Since $v \in \ker(A)$, we have that $Av = 0$. Since $v \in \operatorname{Im}(B)$, we get that there is a $u$ such that $v = Bu$. We therefore get that $ABu = 0$. We can weaken that to $CABu = 0$, and then by equation 4, $BDBu=0$, and then by equation 1, $Bu = 0$, and therefore that $v = 0$ as we intended to show. We conclude that $V = \operatorname{Im}(B) \oplus \ker(A)$.

It can be shown that $V = \operatorname{Im}(A) \oplus \ker(B)$ in the same way.

Let $f : \operatorname{Im}(B) \to \operatorname{Im}(A)$ by defined by $f(v) = Av$ and $g : \ker(A) \to \ker(B)$ be defined by $g(v) = 0$. Given a $u \in w$, it can be written as $u = v + w$ where $v \in \operatorname{Im}(B)$ and $w \in \ker(A)$, we have that $Au = A(v + w) = Av + Aw = f(v) + 0 = f(v) + g(w)$. So we therefore have that $A = f \oplus g$. It remains to show that $f$ is an isomorphism. This follows from the fact that $f$ is clearly surjective and $\operatorname{Im}(A) \cong \operatorname{Im}(B)$ (by lemma \ref{necessary-condition}).

We conclude therefore that $A$ matches the first diagram. The argument for $B$ matching the second diagram is similar, so we don't include it.

We want to show that $C$ matches the third diagram.

We know that $A = f \oplus g$. We want to show that $C = f^{-1} \oplus g^+$. Let $F : \operatorname{Im}(A) \to \operatorname{Im}(B)$ be defined by $F(v) = Cv$. We are allowed to let the codomain of $F$ be $\operatorname{Im}(B)$ because of equation 4: $BD = CA$ implies that $\operatorname{Im}(B) \supseteq C(\operatorname{Im}(A)) = \operatorname{Im}(F)$, which justifies giving $F$ its codomain. Let $G : \ker(B) \to \ker(A)$ be defined by $G(v) = Cv$. We are allowed to let the codomain of $G$ be $\ker(A)$ because of equation 3: Let $v$ be an element of $\ker(B)$; we have that $DBv = ACv$ which implies that $0 = A(G(v))$ which implies that $\operatorname{Im}(G) \subseteq \ker(A)$, which justifies giving $G$ its codomain.

The above shows that $C = F \oplus G$. We now aim to show that $F = f^{-1}$ and $G(v) = 0$ for every $v$.

We want to show that $F = f^{-1}$. It suffices to show that $f(F(v)) = v$ for every $v \in \operatorname{Im}(A)$. Let $v \in \operatorname{Im}(A)$. We have that there is a $u$ such that $v = Au$. We thus have (by equation 1) that $f(F(v)) = f(F(Au)) = ACAu = Au = v$. This is what we intended to show.

We now show that $G(v) = 0$ for every $v \in \ker(B)$. Let $v \in \ker(B)$. We have by equation 2 that $G(v) = Cv = CACv$; and since $Cv$ is in $\ker(A)$, we have that $CACv = 0$, so $G(v) = 0$.

We finally conclude that $C = f^{-1} \oplus g^{+}$. The argument for $D$ matching the fourth diagram is the same. We are done.
\end{proof}

We now prove the converse theorem.

\begin{theorem} \label{diagram-implies-pseudoinverse} If $M = [A,B]$ and $K = [C,D]$ match the above four diagrams, then $K = M^+$. \end{theorem}

\begin{proof}

Once again, we expand out the four axioms of the pseudoinverse:

\begin{enumerate}
    \item $ACA = A$ and $BDB = B$
    \item $CAC = C$ and $DBD = D$
    \item $DB = AC$
    \item $BD = CA$
\end{enumerate}

We aim to show that each of these axioms is satisfied in turn.

We begin with axiom 1. Let $v \in \operatorname{Im}(B)$ and $w \in \ker(A)$. We have that $ACA(v + w) = AC(f(v) + g(w)) = AC(f(v)) = A(f^{-1}(f(v))) = Av = Av + Aw = A(v + w)$. Since $ACA(v + w) = A(v + w)$ is true for every $v$ and $w$, we have that $ACA = A$. The argument for showing that $BDB = B$ is identical.

We continue with axiom 2. Let $v \in \operatorname{Im}(A)$ and $w \in \ker(B)$. We have that $CAC(v + w) = CA(f^{-1}(v) + g^+(w)) = CA(f^{-1}(v)) = C(f(f^{-1}(v))) = Cv = f^{-1}(v) + g^+(w) = C(v + w)$. Since $CAC(v+w) = C(v+w)$ is true for every $v$ and $w$, we have that $CAC = C$. The argument for showing that $DBD = D$ is identical.

We continue with axiom 3. Let $v \in \operatorname{Im}(B)$ and $w \in \ker(A)$. We have that $AC(v + w) = A(f^{-1}(v) + g^+(w)) = A(f^{-1}(v)) = f(f^{-1}(v)) = v$ and that $DB(v + w) = D(h(v) + i(w)) = D(h(v)) = h^{-1}(h(v)) = v$. Since $AC(v + w) = DB(v + w)$ is true for every $v$ and $w$, we have that $AC = DB$.

We continue with axiom 4. Let $v \in \operatorname{Im}(A)$ and $w \in \ker(B)$. We have that $CA(v + w) = C(f(v) + g(w)) = C(f(v)) = f^{-1}(f(v)) = v$ and that $BD(v + w) = B(h^{-1}(v) + i^{+}(w)) = B(h^{-1}(v)) = h(h^{-1}(v)) = v$. Since $BD(v+w) = CA(v+w)$ is true for every $v$ and $w$, we have that $BD = CA$.
\end{proof}

We finally prove a necessary and sufficient condition for the
pseudoinverse to exist.

\begin{theorem} The matrix $M = [A,B]$ has a pseudoinverse if and only if $\operatorname{Im}(AB) = \operatorname{Im}(A) \cong \operatorname{Im}(B) = \operatorname{Im}(BA)$. \end{theorem}

\begin{proof}

Lemma \ref{necessary-condition} already shows that if $M$ has a pseudoinverse then it must satisfy those conditions. So we prove the converse statement.

By theorem \ref{diagram-implies-pseudoinverse}, we only need to show that $A$ matches the following diagram:
$$\begin{aligned}
\operatorname{Im}(B) & \xrightarrow[f]{\sim} & \operatorname{Im}(A) \\
\oplus & & \oplus \\
\ker(A) & \xrightarrow[g]{} & \ker(B) \\
\end{aligned}$$
The argument for $B$ is similar, so we don't include it.

Let $f : \operatorname{Im}(B) \to \operatorname{Im}(A)$ be defined by $f(v) = Av$. The codomain is correct because (by lemma \ref{necessary-condition}) $\operatorname{Im}(A) = \operatorname{Im}(AB) = A(\operatorname{Im}(B)) = \operatorname{Im}(f)$. Since $\operatorname{Im}(A) = \operatorname{Im}(f)$, we have that $f$ is an isomorphism.

Let $g : \ker(A) \to \ker(B)$ be defined by $g(v) = 0$.

We now need to show that $V = \operatorname{Im}(B) \oplus \ker(A)$ where $V$ is the vector space we're working over.

First, we need to show that if $v \in \operatorname{Im}(B) \cap \ker(A)$ then $v = 0$. Let $v \in \operatorname{Im}(B) \cap \ker(A)$. We have that there is a $u$ such that $v = Bu$. We thus have that $0 = Av = ABu$. We know that $A$ is one-to-one map over $\operatorname{Im}(B)$, so $ABu = 0$ implies that $Bu = 0$. We therefore have that $v = 0$.

Lastly, we need to show that $\dim(V) = \operatorname{rank}(B) + \operatorname{null}(A)$. By the rank-nullity theorem, we have that $\dim(V) = \operatorname{rank}(A) + \operatorname{null}(A)$ and $\dim(V) = \operatorname{rank}(B) + \operatorname{null}(B)$. Since $\operatorname{rank}(A) = \operatorname{rank}(B)$ (by lemma \ref{necessary-condition}), we have that $\operatorname{null}(A) = \operatorname{null}(B)$. We finally conclude $\dim(V) = \operatorname{rank}(B) + \operatorname{null}(A)$.

Let $v \in \operatorname{Im}(B)$ and $w \in \ker(A)$. We have that $A(v + w) = Av + Aw = f(v) + 0 = f(v) + g(w)$. We are done.
\end{proof}

\begin{corollary} \label{rank-necessary-and-sufficient} The matrix $M = [A,B]$ has a pseudoinverse if and only if $\operatorname{rank}(AB) = \operatorname{rank}(A) = \operatorname{rank}(B) = \operatorname{rank}(BA)$. \end{corollary}

\hypertarget{relationship-between-the-pseudoinverse-and-the-jordan-svd}{%
\section{Relationship between the pseudoinverse and the Jordan
SVD}\label{relationship-between-the-pseudoinverse-and-the-jordan-svd}}

The following propositions establish that if the pseudoinverse of a
double-complex matrix \(M\) exists then so does the Jordan SVD. It also
supplies a method for finding the Jordan SVD of a matrix under those
conditions. This method works even when the value of the pseudoinverse
isn't known. It therefore gives us a way to \emph{find} the
pseudoinverse if we believe that one exists. The method is to find the
Jordan SVD of \(M\), which is \(M = U[J,J]V^*\), and then conclude that
\(M^+ = V [J^+,J^+] U^*\). This is an adaptation of a common method for
finding the pseudoinverse of a real or complex matrix. We add that assuming that our assumption that $M$ has a pseudoinverse is true, the Jordan matrix $J$ will not have any non-trivially nilpotent Jordan blocks.

Note that there are matrices which have Jordan SVDs but
which don't have pseudoinverses. An example is \([J,J]\) where
\(J = \begin{bmatrix}0 & 1 \\ 0 & 0\end{bmatrix}\).

\begin{lemma} \label{square-root} If $M = [A,B]$ has a pseudoinverse then $\sqrt{AB}$ exists. Additionally, we can pick this square root in such a way that its Jordan Normal Form has no non-trivially nilpotent Jordan blocks. \end{lemma}

\begin{proof}

Since $[A,B]$ has a pseudoinverse, both $A$ and $B$ match the required diagrams. A diagram for $AB$ can be drawn by combining the diagrams for $A$ and $B$:

$$\begin{aligned}
\operatorname{Im}(A) & \xrightarrow[h]{\sim} & \operatorname{Im}(B) & \xrightarrow[f]{\sim} & \operatorname{Im}(A)\\
\oplus & & \oplus & & \oplus \\
\ker(B) & \xrightarrow[i]{} & \ker(A) & \xrightarrow[g]{} & \ker(B)\\
\end{aligned}$$

In other words, we have that $AB(v + w) = f(h(v)) + g(i(w))$ for any $v \in \operatorname{Im}(A)$ and $w \in \ker(B)$. Since $f \circ h$ is an automorphism, and all automorphisms have square roots (by appeal to the Jordan Normal Form) we have that $\sqrt{f \circ h}$ exists. Similarly, since $g \circ i$ is the zero map over $\ker(B)$, we have that $\sqrt{g \circ i}$ exists, and equals $g \circ i$. Therefore we have that $\sqrt{AB}(v + w) = \sqrt{f \circ h}(v)$.

The above definition of $\sqrt{AB}$ is actually one of possibly many square roots of $AB$. This one has the property that its Jordan Normal Form has no non-trivially nilpotent Jordan blocks. We prove this below.

Find a Jordan basis for $\sqrt{f \circ h}$, which we will denote as $B_{\sqrt{f \circ h}}$. Find a Jordan for $g \circ i$, which we will denote as $B_{g \circ i}$. We have that $B_{\sqrt{f \circ h}} \cup B_{g \circ i}$ is a Jordan basis for $\sqrt{AB}$. Since $\sqrt{f \circ h}$ is an automorphism, its JNF won't have any non-trivially nilpotent Jordan blocks. Since $g \circ i$ is the zero map, its JNF also hasn't got any non-trivially nilpotent Jordan blocks. We are done.
\end{proof}

One of the steps in finding the Jordan SVD of a matrix (under the
assumption that a matrix has a pseudoinverse) is to extend an
orthonormal set of vectors \(S\) to an orthonormal basis. Classical
linear algebra (over fields) suggests a way to do this. The procedure is
to pick a random vector \(w\) and perform the \emph{Gram-Schmidt
process} to produce a vector \(w'\) which is orthogonal to every vector
in \(S\). For our purposes, we define the Gram-Schmidt procedure to be
\(w' = w - \sum_{v \in S} \langle v, w \rangle v\). The procedure works
as long as \(w'\) does not have zero norm. If hopefully \(w'\) has
non-zero norm, then we can normalise it and add it to our set \(S\). In
classical linear algebra, it suffices to choose a \(w\) which is not in
the span of \(S\), but over the double-complex numbers this is not sufficient
because a non-zero vector can have zero norm. We deal with this
complication by showing that a \emph{random} choice of \(w\) will almost
surely result (with probability \(1\)) in a \(w'\) which will have
non-zero norm. In particular, this implies that a good choice of \(w'\)
must exist. The probability distribution which we use to pick \(w\) can
have as its support the entire space. This suggests a feasible (if
arguably crude) algorithm for finding \(w'\):

\begin{enumerate}
\def\labelenumi{\arabic{enumi}.}

\item
  Pick a random vector \(w\) using any probability distribution whose
  support is the entire space.
\item
  Perform Gram-Schmidt to produce a vector \(w'\) which is orthogonal to
  every vector in \(S\). Concretely, let
  \(w' = w - \sum_{v \in S} \langle v, w\rangle v\).
\item
  If \(w'\) has zero norm (which has probability \(0\) of happening)
  then go back to step 1.
\item
  Normalise \(w'\) and finish.
\end{enumerate}

Step 3 states that if the choice of \(w\) is a bad one (which should be
apparent after performing step 2) then a different random \(w\) should
be chosen. The probability of needing to perform this repetition is
\(0\).

Our claim that \(w'\) will almost surely have non-zero norm can be
justified using simple algebraic geometry. This is done in the proof
below:

\begin{lemma} \label{orthonormal}

If $\{u_1, u_2, \dotsc, u_n\}$ is an orthonormal set of $d$-dimensional double-complex vectors where $d \geq n$, then this can be extended to form an orthonormal basis $\{u_1, u_2, \dotsc, u_n, u_{n+1}, \dotsc, u_d\}$.

\end{lemma}

\begin{proof}

We will show how to extend the set by one vector. First, choose a vector $w$ and apply the Gram-Schmidt process to make it orthogonal to vectors $u_1$ up to $u_n$. Call the resulting vector $w'$. We must have that the square norm of $w'$ is non-zero. This is important because there are non-zero vectors over the double-complex numbers with zero norm (consider $(1,j)^T$ for instance). If $w'$ has non-zero norm then we may normalise it to a unit vector and add that to our set.

It suffices to show that some $w$ exists for which $w'$ has non-zero norm. We can in fact show a stronger result that almost all randomly chosen ``$w$''s will result in a $w'$ with non-zero norm. To do this, we resort to a simple principle in algebraic geometry. This principle states that the only multi-variate polynomial (for instance, denoted by $f(x,y,z)$) which is zero over a set of non-zero measure is the zero polynomial (\cite{roots}). The situation in which $w'$ has zero norm can be characterised by $w$ satisfying a polynomial equation.

First, we define $w'$ by $w' = w - \sum_{i=1}^n \langle w, u_i \rangle u_i$. The situation in which $w'$ is ``bad'' is when $\langle w', w'\rangle = 0$. We substitute our definition of $w'$ and simplify:

$$\begin{aligned}
&\langle w', w' \rangle = 0\\
\iff &\langle w, w \rangle - \sum_{i=1}^n |\langle w, u_i\rangle|^2 = 0\\
\iff & \sum_{i=1}^d |w_i|^2 - \sum_{i=1}^n \left| \sum_{j=1}^d w_j u_{ij}^* \right|^2 = 0
\end{aligned}$$

We then write $w_i = w_i^{(1)} \frac{1+j}{2} + w_i^{(2)} \frac{1-j}{2}$, $u_{ij} = u_{ij}^{(1)} \frac{1+j}2 + u_{ij}^{(2)} \frac{1-j}{2}$ and $u_{ij}^* = u_{ij}^{(2)} \frac{1+j}{2} + u_{ij}^{(1)} \frac{1-j}{2}$. We can then further simplify the above:

$$\begin{aligned}
& \sum_{i=1}^d (w_i^{(1)} w_i^{(2)}) - \sum_{i=1}^n \left[ \left(\sum_{j=1}^d w_j^{(1)} u_{ij}^{(2)}\right) \left( \sum_{k=1}^d w_k^{(2)} u_{ik}^{(1)} \right) \right] = 0 \\
\iff & \sum_{i=1}^d (w_i^{(1)} w_i^{(2)}) - \sum_{i=1}^n \sum_{j=1}^d \sum_{j=1}^d w_j^{(1)} u_{ij}^{(2)} w_k^{(2)} u_{ik}^{(1)} = 0 \\
\iff & \sum_{i=1}^d (w_i^{(1)} w_i^{(2)}) - \sum_{j=1}^d \sum_{j=1}^d \left(w_j^{(1)} w_k^{(2)} \left(\sum_{i=1}^n u_{ij}^{(2)} u_{ik}^{(1)} \right) \right) = 0 \\
\iff & \sum_{j=1}^d \sum_{k=1}^d \left( w_j^{(1)} w_k^{(2)} \left(\delta_{jk} - \sum_{i=1}^n u_{ij}^{(2)} u_{ik}^{(1)}\right)\right) = 0
\end{aligned}$$

where $\delta_{jk}$ denotes the Kronecker delta symbol. Imagine that the polynomial on the left hand side is the zero polynomial. This is equivalent to $\delta_{jk} - \sum_{i=1}^n u_{ij}^{(2)} u_{ik}^{(1)} = 0$ being true for every $j, k \in \{1,2,\dotsc,d\}$. This can be expressed as a matrix equation $(U^{(2)})^T U^{(1)} = I_d$. The matrices $U^{(2)}$ and $U^{(1)}$ have dimension $n \times d$. The matrix equation cannot be satisfied because $\operatorname{rank}(I_d) = d > n \geq \operatorname{rank}(U^{(2)})$, so the rank of one side of the equation is not equal to the rank of the other.
\end{proof}

\begin{lemma} \label{block-matrix-condition}

A matrix of the form $L \oplus M$ (where ``$\oplus$'' means direct sum of matrices) has a pseudoinverse if and only if $L$ and $M$ have pseudoinverses. In which case, $(L \oplus M)^+ = L^+ \oplus M^+$. This is true for arbitrary rings.

\end{lemma}

\begin{proof}

We haven't found a proof of this basic fact in any textbook, so we provide one here. The proof uses only the four defining axioms, making it valid over any ring.

In one direction, the proof is obvious. If $L$ and $M$ have pseudoinverses then $L^+ \oplus M^+$ must be the pseudoinverse of $L \oplus M$.

We prove the converse. Write $(L \oplus M)^+$ as a block matrix: $(L \oplus M)^+ = \begin{bmatrix}A & B \\ C & D \end{bmatrix}$. Notice that the aim is to show that $A = L^+$ and $D = M^+$ (or equivalently, that $B$ and $C$ are zero matrices).

Plugging the block matrix representations of $L \oplus M$ and $(L \oplus M)^+$ into the 4 axioms of the pseudoinverse gives:

\begin{enumerate}
  \item $\left[\begin{matrix}L A L & L B M\\M C L & M D M\end{matrix}\right] = \left[\begin{matrix}L & 0\\0 & M\end{matrix}\right]$
  \item $\left[\begin{matrix}A L A + B M C & A L B + B M D\\C L A + D M C & C L B + D M D\end{matrix}\right] = \left[\begin{matrix}A & B\\C & D\end{matrix}\right]$
  \item $\begin{bmatrix}(LA)^* &(MC)^* \\ (LB)^* & (MD)^*\end{bmatrix} = \begin{bmatrix} LA & LB \\ MC & MD\end{bmatrix}$
  \item $\begin{bmatrix} (AL)^* & (CL)^* \\ (BM)^* & (DM)^* \end{bmatrix} = \begin{bmatrix} AL & BM \\ CL & DM\end{bmatrix}$
\end{enumerate}

This shows that $A$ satisfies all of the equations to be the pseudoinverse of $L$, except for axiom 2 ($ALA = A$). Likewise, the same can be said for the relationship between $D$ and $M$ (all the conditions follow except for axiom 2: $DMD = D$).

The following observation is based on Section 3.1 of \emph{Theory of generalized inverses over commutative rings} by K.P.S. Bhaskara Rao. There, it shows that if a matrix has a generalised inverse satisfying axioms 1, 3 and 4, then it has a pseudoinverse.

Alternatively, it can be directly shown (via the axioms) that $ALA$ is a pseudoinverse of $L$. Likewise, $DMD$ is the pseudoinverse of $M$. Both claims are easily verified.
\end{proof}

\begin{lemma} \label{pseudoinverse-jordan-blocks}

Let $S = [J_1,J_1] \oplus [J_2,J_2] \oplus \dotsb \oplus [J_k,J_k]$ where each $J_i$ is a Jordan block. $S$ has a pseudoinverse if and only if every $J_i$ is invertible or equal to $[0]$. In which case, $S^+ = [J_1^+,J_1^+] \oplus [J_2^+,J_2^+] \oplus \dotsb \oplus [J_k^+,J_k^+]$

\end{lemma}

\begin{proof}

By lemma \ref{block-matrix-condition}, we have that $S$ only has a pseudoinverse if and only if each $[J_i,J_i]$ has a pseudoinverse. In which case, we have that $S^+ = [J_1,J_1]^+ \oplus [J_2,J_2]^+ \oplus \dotsb \oplus [J_k,J_k]^+$. Clearly, if $J_i$ is invertible then $[J_i,J_i]^+ = [J_i^{-1},J_i^{-1}] = [J_i^{+},J_i^{+}]$. If $J_i$ is non-invertible with $J_i = [0]$, then we can use the fact that $[0]^+ = [0]$ (easily verified) to conclude that $[J_i, J_i]^+ = [J_i^+, J_i^+]$. This then leaves the case where $J_i$ is a non-trivial nilpotent, in which case $[J_i,J_i]$ does not have a pseudoinverse because the condition $\operatorname{Im}(AB) = \operatorname{Im}(A) \cong \operatorname{Im}(B) = \operatorname{Im}(BA)$ given in lemma \ref{necessary-condition} is violated.
\end{proof}

\begin{lemma} \label{kernel-pinv}
  If $M$ has a pseudoinverse then $\ker(M^* M) = \ker(M)$.
\end{lemma}
\begin{proof}
  Let $M = [A,B]$. Clearly, $\ker(M) \subseteq \ker(M^* M)$. We wish to conclude that $\ker(M) = \ker(M^* M)$. This is equivalent to saying that $\ker(A) = \ker(BA)$ and $\ker(B^T) = \ker(A^T B^T)$. We can argue that this is true by showing that $\operatorname{rank}(A) = \operatorname{rank}(BA) = \operatorname{rank}(B^T) = \operatorname{rank}(A^T B^T)$. This then follows from corollary \ref{rank-necessary-and-sufficient}.
\end{proof}

\begin{theorem} If the matrix pseudoinverse exists, then so does the Jordan SVD. \end{theorem}

\begin{proof}

We will show how to obtain the Jordan SVD of a matrix $M$ by making use of the pseudoinverse operation. The Jordan SVD of $M$ is written $U S V^*$.

We will first show how to obtain $V$ and $S$. Let $M = [A,B]$. Consider $\sqrt{M^* M} = [\sqrt{BA},\sqrt{BA}]$. By lemma \ref{square-root}, we know that $\sqrt{BA}$ exists. We take its Jordan decomposition and get $\sqrt{BA} = PJP^{-1}$. Let $V = [P,P^{-1}]$ and $S = [J,J]$. We have obtained $V$ and $S$, as we set out to do.

It remains now to find $U$. Since we aim to satisfy the equation $M = USV^*$, it's reasonable to think that $MVS^+$ might equal $U$ (where we have used the pseudoinverse operation on $S = [J,J]$). In general, it does not, but it's a reasonable starting point.

What is $S^+$ anyway, and how do we know that it exists? By lemma \ref{square-root}, we know that the matrix $J$ has not got any non-trivially nilpotent Jordan blocks. By lemma \ref{pseudoinverse-jordan-blocks}, $S = [J,J]$ must therefore have a pseudoinverse, which equals $[J^+,J^+]$.

Let $U' = MVS^+$. We aim to show that the non-zero columns of $U'$ form a pairwise orthogonal set. If this is true, then we can construct the matrix $U$ from $U'$ by replacing the zero columns of $U'$ with column vectors that are orthogonal to the non-zero columns of $U'$. A minor complication is that a non-zero column might still have zero norm, but this possibility is refuted later on. Lemma \ref{orthonormal} justifies that without this complication, we can always do this. One way to show that the columns of $U'$ are pairwise orthogonal is to show that $(U')^* U'$ is a diagonal matrix. Since $U' = MVS^+ = [APJ^+,J^+P^{-1}B]$, we have that $(U')^* U' = [J^+ P^{-1}BAPJ^+,J^+ P^{-1}BAPJ^+]$. Obviously, the two components of $(U')^* U'$ are identical because the matrix must be Hermitian. It remains to see that $J^+ P^{-1}BAPJ^+$ is a diagonal matrix. We simplify $J^+ P^{-1}BAPJ^+$ to get $J^+ P^{-1}BAPJ^+ = J^+ P^{-1}(PJP^{-1})^2PJ^+ = J^+ P^{-1}PJ^2P^{-1}PJ^+ = J^+ J^2J^+$. It remains to see that $J^+ J^2J^+$ is a diagonal matrix. By analysing the Jordan block composition of $J$ (which we get from lemma \ref{pseudoinverse-jordan-blocks}), we get that $J^+ J^2 J^+$ has the layout $\begin{bmatrix} 1 & & & & \\ & 1 & & & \\ & & \ddots & & \\ & & & 0 & \\ & & & & 0\end{bmatrix}$.

We need to rule out a complication: What if some column $c$ of $U'$ has norm $0$? We need to be able to conclude that $c$ is a zero column in that case. The following argument shows that this is indeed the case: We have that $S = [J_1, J_1] \oplus [J_2, J_2] \oplus \dotsb \oplus [J_k, J_k]$, so $S^+ = [J_1^+, J_1^+] \oplus [J_2^+, J_2^+] \oplus \dotsb \oplus [J_k^+, J_k^+]$. The column $c$ must be positioned in $U'$ in the same place as some block $[J_i^+, J_i^+]$ is positioned in $S^+$.\footnote{To explain what we mean: Let's say that $c$ is the second column of $U'$. Say that the second block $[J_2^+,J_2^+]$ of $S^+$ occupies columns 2 to 4 in matrix $S^+$. Since $c$ is found in column 2, and $[J_2^+,J_2^+]$ is also found in column 2, we can say that they are ``positioned in the same place'' in $U'$ and $S^+$ respectively.} $J_i$ is a non-invertible Jordan block.\footnote{To explain why: If $c$ occupies column $i$ of $U'$, we must then have that the $ii$th entry of $(U')^*U'$ is $0$ (the norm of $c$). But we also have that $(U')^*U' = J^+ J^2 J^+$. Let $[J_i,J_i]$ be positioned in the same place in $S$ as $c$ is in $U'$. If $J_i$ is invertible then $J_i^+ J_i^2 J_i^+ = I$, which would imply that the norm of $c$ is $1$, which it is not.} We know that $J_i$ cannot be non-trivially nilpotent, so we conclude that $J_i = [0]$. Recalling the definition $U' = MVS^+$, the multiplication of $MV$ on the right by $S^+$ causes the corresponding column of $MV$ to be scaled by $0$. So $c$ is a zero column.

We now need to show that if we multiply $U' SV^*$, we do indeed get $M$. We then leave verifying the fact $U S V^*$ is equal to $M$ to the next pagraph. Let $K = U' S V^*$. We wish to show that $K = M$. We have that $K = U' S V^* = MVS^+ S V^* = MV (J^+ J^2 J^+) V^* = M V \operatorname{diag}(1,1,\dotsc,0,0) V^*$. Observe that the columns of $V$ form a basis. Call the columns $v_1, v_2, \dotsc, v_n$. We wish to show that $Mv_i = Kv_i$ for every $v_i$; from which we can conclude that $M = K$. We note that $M^* M = V [J^2,J^2] V^*$. Either $M^* Mv_i \neq 0$ or $M^* Mv_i = 0$. If $M^* M v_i \neq 0$, we have that column $v_i$ in $V$ is positioned in the same place as Jordan block $J_j$ in $J$ where $J_j$ is invertible; we therefore have that $Kv_i = M V \operatorname{diag}(1,1,\dotsc,0,0) V^* v_i = M V \operatorname{diag}(1,1,\dotsc,0,0) e_i = M V e_i = Mv_i$. If on the other hand, we have that $M^* M v_i = 0$, then we have that $Mv_i = 0$ by lemma \ref{kernel-pinv}; we also have that column $v_i$ in $V$ is positioned in the same place as $J_j$ in $J$ where $J_j = [0]$, so we that $Kv_i = MV \operatorname{diag}(1,1,\dotsc,0,0)V^* v_i = 0 = Mv_i$.

Now that we know that $M = U' S V^*$, we wish to also conclude that $M = U S V^*$. Either a column $u_i$ of $U$ is found in $U'$, or it's positioned in the same place as a zero column in $U'$. If $u_i$ is in $U$, then the reader can verify that $U' S V^* v_i = U' S e_i = U S e_i = U S V^* v_i$. If $u_i$ is not in $U$, then we have that $Mv_i = U' S V^* v_i = U' S e_i = U' 0 = 0 = U S e_i = U S V^* v_i$.
\end{proof}

\hypertarget{conclusion}{%
\section{Conclusion}\label{conclusion}}

We finish this paper with two open problems concerning the Jordan SVD.
The Jordan SVD can be defined without using double or double-complex
numbers, allowing the broader linear algebra community to suggest
solutions to these problems. We say that a pair of \(n \times n\)
complex matrices \(A\) and \(B\), written \([A,B]\), has a Jordan SVD if
there exist invertible matrices \(P\) and \(Q\), and a Jordan matrix
\(J\), such that:

\begin{itemize}

\item
  \(A = PJQ^{-1}\).
\item
  \(B = QJP^{-1}\).
\end{itemize}

The two open problems are:

\begin{enumerate}
\def\labelenumi{\arabic{enumi}.}

\item
  Find a necessary and sufficient condition for the Jordan SVD to exist.
\item
  Prove that the Jordan matrix \(J\) is unique in all cases, subject to
  the restriction that all the eigenvalues of \(J\) belong to the
  \emph{half-plane}. In \cite{gutin}, we defined the half-plane to consist of
  those complex numbers which either have positive real part, or which
  have real part equal to zero but which have non-negative imaginary part.
\end{enumerate}

In this paper, we have made progress on problem 1 by showing that a
sufficient condition for the Jordan SVD to exist is that
\(\operatorname{rank}(A) = \operatorname{rank}(B) = \operatorname{rank}(AB) = \operatorname{rank}(BA)\).
Still, there are matrix pairs \([A,B]\) which have Jordan SVDs but which
don't satisfy this condition. 
The paper \cite{kaplansky} may be relevant, and taking inspiration from
it one might conjecture that a necessary and sufficient condition is for
$AB$ to be similar to $BA$.

In \cite{gutin}, we made progress on
problem 2 by showing that the Jordan SVD is unique whenever \(A\) and
\(B\) are invertible.

\bigskip{\bf Acknowledgments.} I would like to thank Gregory Gutin for many useful comments and suggestions. I would like to credit Sympy (\cite{sympy}) for its aid in helping me discover and prove results. I would like to thank the referee for his or her valuable feedback.

\bibliographystyle{unsrt}
\bibliography{paper2}

\end{document}